\documentclass[11pt]{article} 

\usepackage[utf8]{inputenc} 

\usepackage{geometry} 
\usepackage{graphicx} 
\usepackage{color}

\definecolor{Red}{rgb}{1,0,0}
\definecolor{Blue}{rgb}{0,0,1}
\definecolor{Olive}{rgb}{0.41,0.55,0.13}
\definecolor{Green}{rgb}{0,1,0}
\definecolor{MGreen}{rgb}{0,0.8,0}
\definecolor{DGreen}{rgb}{0,0.55,0}
\definecolor{Yellow}{rgb}{1,1,0}
\definecolor{Cyan}{rgb}{0,1,1}
\definecolor{Magenta}{rgb}{1,0,1}
\definecolor{Orange}{rgb}{1,.5,0}
\definecolor{Violet}{rgb}{.5,0,.5}
\definecolor{Purple}{rgb}{.75,0,.25}
\definecolor{Brown}{rgb}{.75,.5,.25}
\definecolor{Grey}{rgb}{.5,.5,.5}

\usepackage{booktabs} 
\usepackage{array} 
\usepackage{paralist} 
\usepackage{verbatim} 
\usepackage{subfigure} 
\usepackage{amsmath,amssymb,amsthm,mathrsfs}
\usepackage[colorlinks]{hyperref}
\usepackage{blkarray}

\usepackage{fancyhdr} 
\makeatletter
\newtheorem*{rep@theorem}{\rep@title}
\newcommand{\newreptheorem}[2]{%
\newenvironment{rep#1}[1]{%
\def\rep@title{#2 \ref{##1}}%
\begin{rep@theorem}}%
{\end{rep@theorem}}}
\makeatother

\theoremstyle{plain}

\newtheorem{theorem}{Theorem}

\newtheorem{theorem*}{Theorem}   
\newtheorem{lemma*}{Lemma} 
\newtheorem{corollary*}{Corollary} 
\newtheorem*{remark*}{Remark}
\newtheorem{example}{Example}

\newlength{\widebarargwidth}
\newlength{\widebarargheight}
\newlength{\widebarargdepth}

\theoremstyle{definition}
\newtheorem{definition}{Definition}

\def\cC{{\cal C}}

\def\cE{{\cal E}}

\newcommand{\defn}{\ensuremath{:  =}}

\newcommand{\E}{\ensuremath{\mathbb{E}}}

\newcommand{\mprob}{\ensuremath{\mathbb{P}}}

\begin{document}

\begin{center}

{\bf{\Large{Analysis of centrality in sublinear preferential attachment trees via the Crump-Mode-Jagers branching process}}}

\vspace*{.25in}

\begin{tabular}{ccc}
{\large{Varun Jog}} & \hspace*{.5in} & {\large{Po-Ling Loh}} \\
{\large{\texttt{vjog@ece.wisc.edu}}} & & {\large{\texttt{loh@ece.wisc.edu}}} \vspace{.2in}
\end{tabular}

Department of ECE \\
Grainger Institute for Engineering \\
University of Wisconsin - Madison \\ Madison, WI 53715

\vspace*{.2in}

October 2016

\vspace*{.2in}

\end{center}

\begin{abstract}
We investigate centrality and root-inference properties in a class of growing random graphs known as sublinear preferential attachment trees. We show that a continuous time branching processes called the Crump-Mode-Jagers (CMJ) branching process is well-suited to analyze such random trees, and prove that almost surely, a unique terminal tree centroid emerges, having the property that it becomes more central than any other fixed vertex in the limit of the random growth process. Our result generalizes and extends previous work establishing persistent centrality in uniform and linear preferential attachment trees. We also show that centrality may be utilized to generate a finite-sized $1-\epsilon$ confidence set for the root node, for any $\epsilon> 0$, in a certain subclass of sublinear preferential attachment trees.
\end{abstract}


\section{Introduction}

Recent years have seen an explosion of datasets possessing some form of underlying network structure~\cite{DorMen03, BarEtal08, Jack08, New10}. Various mathematical models have consequently been derived to imitate the behavior of real-world networks; desirable characteristics include degree distributions, connectivity, and clustering, to name a few. One popular probabilistic model is the Barab\'{a}si-Albert model, also known as the (linear) preferential attachment model~\cite{Barabasi99}. Nodes are added to the network one at a time, and each new node connects to a fixed number of existing nodes with probability proportional to the degrees of the nodes. In addition to modeling a ``rich get richer" phenomenon, the Barab\'{a}si-Albert model gives rise to a scale-free graph, in which the degree distribution in the graph decays as an inverse polynomial power of the degree, and the maximum degree scales as the square root of the size of the network. Such a property is readily observed in many network data sets~\cite{AlbBar02}.

However, networks also exist in which the disparity between high- and low-degree nodes is not as severe. In the sublinear preferential attachment model, nodes are added sequentially with probability of attachment proportional to a fractional power of the degree. This leads to a stretched exponential degree distribution and a maximum degree that scales as a power of the logarithm of the number of nodes~\cite{KraEtal00, Bar16}. Networks exhibiting such behavior include certain citation networks, Wikipedia edit networks, rating networks, and the Digg network~\cite{KunEtal13}. The case when the probability of attachment is uniform over existing vertices is known as uniform attachment and is used to model networks in which the preference given to older nodes is attributed only to birth order and not degree.

The iterative nature of the preferential attachment model generates interesting questions concerning phenomena that arise (and potentially vanish) as the network expands. Dereich and M\"{o}rters~\cite{DerMoe09} established the emergence of a persistent hub---a vertex that remains the highest-degree node in the network after a finite amount of time---in a certain preferential attachment model where edges are added independently. Such a result was also shown to hold for the Barab\'{a}si-Albert preferential attachment model in Galashin~\cite{Galashin13}. Motivated by the fact that persistent hubs do not exist in uniform attachment models, however, our previous work~\cite{JogLoh15} studied the problem of persistent centroids and established that the $K$ most central nodes according to a notion of ``balancedness centrality" always persist in preferential and uniform attachment trees.

Another related problem concerns identifying the oldest node(s) in a network. Shah and Zaman~\cite{Shah11} first studied this problem in the context of a random growing tree formed by a diffusion spreading over a regular tree, and showed that the centroid of the diffusion tree agrees with the root node of the diffusion, with strictly positive probability. Bubeck et al.~\cite{Bubeck14} devised confidence set estimators for the first node in preferential and uniform attachment trees, in which the goal is to identify a set of nodes containing the oldest node, with probability at least $1-\epsilon$. They showed that when nodes are selected according to an appropriate measure of ``balancedness centrality," the required size of the confidence set is a function of $\epsilon$ that does not grow with the overall size of the network. These results were later extended to diffusions spreading over regular trees by Khim and Loh~\cite{KhimLoh15}. Graph centrality ideas, in particular balancedness centrality, have also been leveraged in Tan et al.~\cite{TanEtAl16} to identify the most influential vertices in a social network.  Luo et al.~\cite{LuoEtal13} studied the problem of identifying single or multiple sources of rumors in a graph and proposed certain efficiently computable estimators related to the MAP estimator employed in Shah and Zaman~\cite{Shah11}. Recently, rumor identification has also been analyzed in certain probabilistic models, such as repeated observations of rumor spreading in Dong et al.~\cite{Dong13}, and incomplete information about rumor spreading in Karamchandani et al.~\cite{Karam13}. In addition to having obvious practical implications for pinpointing the origin of a network based on observing a large graph, identifying and removing the oldest nodes may have desirable deleterious effects from the point of view of network robustness~\cite{EckMoe14}.

Previous analysis of determining a finite confidence set \cite{Bubeck14, KhimLoh15}, as well as establishing the persistence of a unique tree centroid \cite{JogLoh15}, crucially depended on the following property satisfied by linear preferential attachment, uniform attachment, and diffusions over regular trees: the ``attraction function" relating the degree of a vertex to its probability of connection at each time step is linear. Bubeck et al.~\cite{Bubeck14} posed an open question concerning the existence of finite-sized confidence sets in the case of sublinear or superlinear preferential attachment; we likewise conjectured in previous work that a unique centroid should persist for a more general class of nonlinear attraction functions~\cite{JogLoh15}. However, the techniques in these papers do not extend readily to nonlinear settings. An approach to dealing with more complicated tree models in the context of diffusions was presented in Shah and Zaman~\cite{Shah15}, using a continuous time branching process known as the Bellman-Harris branching process. In this paper, we show that preferential attachment trees with nonlinear attraction functions may also be analyzed via continuous time branching processes. Our results rely on properties of the Crump-Mode-Jagers (CMJ) branching process~\cite{CrumpMode68, CrumpMode69,Jagers75}. Continuous time branching processes were previously leveraged by Bhamidi~\cite{Bhamidi07} and Rudas et al.~\cite{Rudas07} to establish properties regarding the degree distribution, maximum degree, height, and local structure of a large class of preferential attachment trees.

Our main contributions are twofold: First, we establish the property of \emph{terminal centrality} for sublinear preferential attachment trees, thereby addressing our conjecture in \cite{JogLoh15}. We prove the existence of a unique vertex that becomes more central than any other vertex, in the limit of the growth process. In fact, the existence of a persistent centroid implies terminal centrality, but the latter implication might not hold, since persistent centrality requires a tree centroid to emerge and remain the centroid starting from a single finite time point.  Second, we affirmatively answer the open question of Bubeck et al.~\cite{Bubeck14} by devising finite-sized confidence sets for the root node in sublinear preferential attachment trees. Due to the inapplicability of P\'{o}lya urn theory in the present setting, the proof techniques employed in our paper differ significantly from the analysis used in previous work. Furthermore, the literature concerning CMJ branching processes is vast and unconsolidated, and another important technical contribution of our paper is to gather relevant results and show that they may be applied to study sublinear preferential attachment trees.


The remainder of the paper is organized as follows: In Section~\ref{section: CMJ}, we review CMJ branching processes and show how to embed a preferential attachment tree in a CMJ process. We also verify that the CMJ processes corresponding to certain sublinear preferential attachment trees enjoy useful convergence properties. In Section~\ref{section: persistence}, we establish the existence of a unique terminal centroid in sublinear preferential attachment trees. In Section~\ref{section: adam}, we prove that the confidence set construction via the same centrality measure leads to finite-sized confidence sets for the root node. Although we believe sublinear preferential attachment trees should also possess a persistent centroid, some challenges arise in bridging the gap between terminal centrality and persistent centrality. We discuss these challenges and related open problems in Section~\ref{section: discussion}. Additional proof details are contained in the supplementary appendices. \\

\textbf{Notation:} We write $V(T)$ to denote the set of vertices of a tree $T$, and write Max-Deg$(T)$ to denote the maximum degree of the vertices in $T$. For $u \in V(T)$, we write $(T,u)$ to denote the corresponding rooted tree, which is a tree with directed edges emanating from $u$. We write $(T,u)_{v \downarrow}$ to denote the subtree directed away from $u$ and starting from $v$. Finally, we write Out-Deg$(v)$ to denote the number of children of vertex $v$ in the rooted tree.

\section{Preliminaries}\label{section: CMJ}

In this section, we review properties of the CMJ branching process, laying the groundwork for our analysis of sublinear preferential attachment trees. The CMJ branching process is a general age-dependent continuous time branching process model introduced by Crump, Mode, and Jagers~\cite{CrumpMode68, CrumpMode69, Jagers75}. It begins with a single individual, known as the ancestor, at time $t=0$. An individual $x$ may give birth multiple times throughout its lifetime, and the times at which it produces offspring are given by a point process $\xi$ on $\mathbb R_+$. The defining property of branching processes is that individuals behave in an i.i.d.\ manner; i.e., every individual starts its own independent point process of births from the moment it is born until the time it dies. The resulting branching process is said to be driven by $\xi$. Many common branching processes are special cases of a CMJ process with an appropriate point process and lifetime random variable: If individuals have random lifetimes and give birth to a random number of children at the moment of their death, the resulting branching process is called the Bellman-Harris process. If the lifetimes of individuals are also constant (usually taken to be 1), the resulting process is known as the Galton-Watson process \cite{AthNey04, Har12}.



\begin{definition}[Random preferential attachment tree with attraction function $f$]\label{def: PA}
A sequence of random trees $\{T_n\}$ is generated as follows: At time $n=1$, the tree $T_1$ consists of a single vertex $v_1$. At time $n+1$, a new vertex $v_{n+1}$ is added to $T_n$  via a directed edge from a vertex $v_i$ to $v_{n+1}$, where $v_i$ is chosen with probability proportional to $f(\text{Out-Deg}(v_i))$ and $\text{Out-Deg}(v_i)$ is computed with respect to the tree $T_n$.
\end{definition}
Thus, the linear preferential attachment tree corresponds to the attraction function $f(i) = i+1$,\footnote{Note that for all nodes except the root node, $\text{Deg}(v_i) = \text{Out-Deg}(v_i)+1$. Thus, this model differs slightly from the one considered in our previous work~\cite{JogLoh15} and in Bubeck et al.~\cite{Bubeck14}, since the attractiveness of $v_1$ is proportional to $\text{Deg}(v_1)+1$ rather than $\text{Deg}(v_1)$.} and the uniform attachment tree corresponds to the constant function $f \equiv 1$. We now define sublinear preferential attachment trees, which have an attraction function that lies strictly between those of a linear preferential attachment tree and a uniform attachment tree.

\begin{definition}[Sublinear preferential attachment trees]
Sublinear preferential attachment trees are preferential attachment trees with an attraction function $f$ satisfying the following conditions:
\begin{enumerate}
\item
$f$ is a nondecreasing function.
\item
$f(i) \geq 1$ for all $i \geq 0$, and $f$ is not identically equal to 1.
\item
There exists  $0 < \alpha <1$ such that $$f(i) \leq (i+1)^\alpha,$$
for all $i \geq 0$.
\end{enumerate}
\end{definition}
Note that the last condition implies $f(0) = 1$. When $f(i) = (i+)^\alpha$, we denote the corresponding tree to be the $\alpha$-sublinear preferential attachment tree. To define the branching process corresponding to a preferential attachment tree, we define the point process $\xi$ associated with the attraction function $f$:
\begin{definition}[Point process associated to $f$]
Given an attraction function $f$, the associated point process $\xi$ on $\mathbb R_+$ is a pure-birth Markov process with $f$ as its rate function:
$$\mathbb P\left(\xi(t+dt) - \xi(t) = 1 \mid \xi(t)=i\right) = f(i)dt + o(dt),$$
with the initial condition $\xi(0) = 0$.
\end{definition}
Note that we do not need to normalize the rate of this Markov process: Consider a CMJ process driven by the point process $\xi$ as above, in which individuals never die. Suppose that at some time $t_0$, the branching process consists of $n$ individuals $\{v_1, \dots, v_n\}$, where the number of children of node $v_i$ is denoted by $d_i$. In the discrete time tree evolution, the next vertex $v_{n+1}$ attaches to vertex $v_i$ with probability $\frac{f(d_i)}{\sum_{j=1}^n f(d_j)}$. In the continuous time process, the new vertex ``attaches to $v_i$" if and only if node $i$ has a child before any of the other nodes. This child is then $v_{n+1}$. Using properties of the exponential distribution, we may check that this happens with probability $\frac{f(d_i)}{\sum_{j=1}^n f(d_j)}$, which is exactly the same as that in the discrete time tree evolution. Thus, if we look at the CMJ branching process at the stopping times when successive vertices are born, the resulting trees evolve in the same way as in the discrete time model described in Definition~\ref{def: PA}.

\begin{definition}[Malthusian parameter]
For a point process $\xi$ on $\mathbb R_+$, let $\mu(t) = \mathbb E[\xi(0, t]]$ denote the mean intensity measure. The point process $\xi$ is a \emph{Malthusian process} if there exists a parameter $\theta > 0$ such that
\begin{equation*}
\theta \int_{0}^\infty e^{-\theta t}\mu(t) dt = 1.
\end{equation*}
The constant $\theta$ is called the \emph{Malthusian parameter} of the point process $\xi$.
\end{definition}
\begin{example}
\label{ExaPA}
For the linear preferential attachment tree with $f(i) = i+1$, the associated point process $\xi$ is the \emph{standard Yule process}, defined as follows:
\begin{enumerate}
\item[(a)] $\xi(0) = 0$, and
\item[(b)] $\mathbb P\left(\xi(t+dt)-\xi(t) = 1 \mid \xi(t) = i\right) = (i+1)dt + o(dt).$
\end{enumerate}
The mean intensity measure for the Yule process is $\mu(t) = e^t - 1$, and the Malthusian parameter is equal to 2.
\end{example}
\begin{example}
\label{ExaUA}
For the uniform attachment tree with $f \equiv 1$, the associated point process $\xi$ is the \emph{Poisson point process} with rate 1. The mean intensity measure is $\mu(t) = t$, and the Malthusian parameter is equal to 1.
\end{example}

The Malthusian parameter of a point process plays a critical role in the theory of branching processes. It accurately characterizes the growth rate of the population generated by the CMJ branching process driven by the point process, as follows: If the population at time $t$ is given by $Z_t$, the random variable $e^{-\theta t}Z_t$ converges to a nondegenerate random variable $W$. Various assumptions on the point process lead to different types of convergence results, such as convergence in distribution, in probability, almost surely, in $L^1$, or in $L^2$~\cite{CrumpMode68, CrumpMode69, Doney72, Nerman81}. As derived in Lemma~\ref{lemma: malthus} in Appendix~\ref{appendix: cmj}, the Malthusian parameter for a sublinear preferential attachment process always exists and lies between the values corresponding to linear preferential attachment and uniform attachment trees described in Examples~\ref{ExaPA} and~\ref{ExaUA}.

Our results will rely heavily on the following theorem: 

\begin{theorem}\label{thm: ABC}
Let $\xi$ be the point process corresponding to a sublinear attraction function $f$. The CMJ branching process $Z_t$ driven by $\xi$ describing the growing random tree satisfies
$$e^{-\theta t} Z_t \stackrel{L^2, \text{ a.s.}} \longrightarrow W,$$
where $W$ is an absolutely or singular continuous random variable supported on all of $\mathbb R_+$, satisfying $W > 0$, almost surely.
\end{theorem}
The proof of Theorem~\ref{thm: ABC}, which is contained in Appendix \ref{appendix: cmj}, is established by showing that the technical conditions required for certain theorems about CMJ processes~\cite{Nerman81, JagerNerman84, Biggins79} are satisfied by the point process $\xi$.

\section{Terminal centrality}\label{section: persistence}

We now turn to our main result, which establishes the existence of a unique terminal centroid in sublinear preferential attachment trees. We begin by introducing some notation and basic terminology.

Consider the function $\psi_T: V(T) \to \mathbb N$ defined by
\begin{equation*}
\psi_T(u) = \max_{v \in V(T) \setminus \{u\}} |(T, u)_{v\downarrow}|.
\end{equation*}
Recall that $(T, u)_{v \downarrow}$ denotes the subtree of $T$ directed away from $u$, starting at $v$, as depicted in Figure~\ref{FigDownarrow}. Thus, $\psi_T(u)$ is the size of the largest subtree of the rooted tree $(T, u)$, and measures the level of ``balancedness" of the tree with respect to vertex $u$. We make the following definition:
\begin{definition}
\label{DefCentroid}
Given a tree $T$, a vertex $u \in V(T)$ is called a \emph{centroid} if $\psi_T(u) \leq \psi_T(v)$, for all  $v \in V(T)$.
\end{definition}
Note that although we have defined the centroid with respect to the criterion $\psi_T$, numerous equivalent characterizations of tree centroids exist~\cite{Jor1869, Jack08, Zel68, Sla75, Sla81, Mit78}. (The characterization appearing in Definition~\ref{DefCentroid} coincides with the notion of ``rumor center" defined by Shah and Zaman~\cite{Shah15}.) Furthermore, a tree may have more than one centroid (although by Lemma~\ref{lemma: jogloh1} in Appendix~\ref{AppTrees}, a tree may have at most two centroids, which must then be neighbors). For any two nodes $u$ and $v$, if  $\psi_T(u) \leq \psi_T(v)$, we say that $u$ is \emph{at least as central} as $v$. Finally, we define the notion of terminal centrality:
\begin{definition}
A vertex $v^* \in \cup_{n=1}^\infty V(T_n)$ is a \emph{terminal centroid} for the sequence of growing trees $\{T_n\}_{n \ge 1}$ if for every vertex $u \neq v^*$, there exists a time $M$ (possibly dependent on $u$), such that for all times $n \geq M$, we have
$$\psi_{T_n}(v^*) < \psi_{T_n}(u).$$
\end{definition}
Thus, the terminal centroid eventually becomes more central than any other fixed vertex. (Note, however, that terminal centrality does not immediately imply the property of \emph{persistent centrality}; for instance, $v^*$ might be a terminal centroid without ever being the centroid at any finite time.) We have the following theorem:

\begin{figure}
\begin{center}
\includegraphics[scale = 0.5]{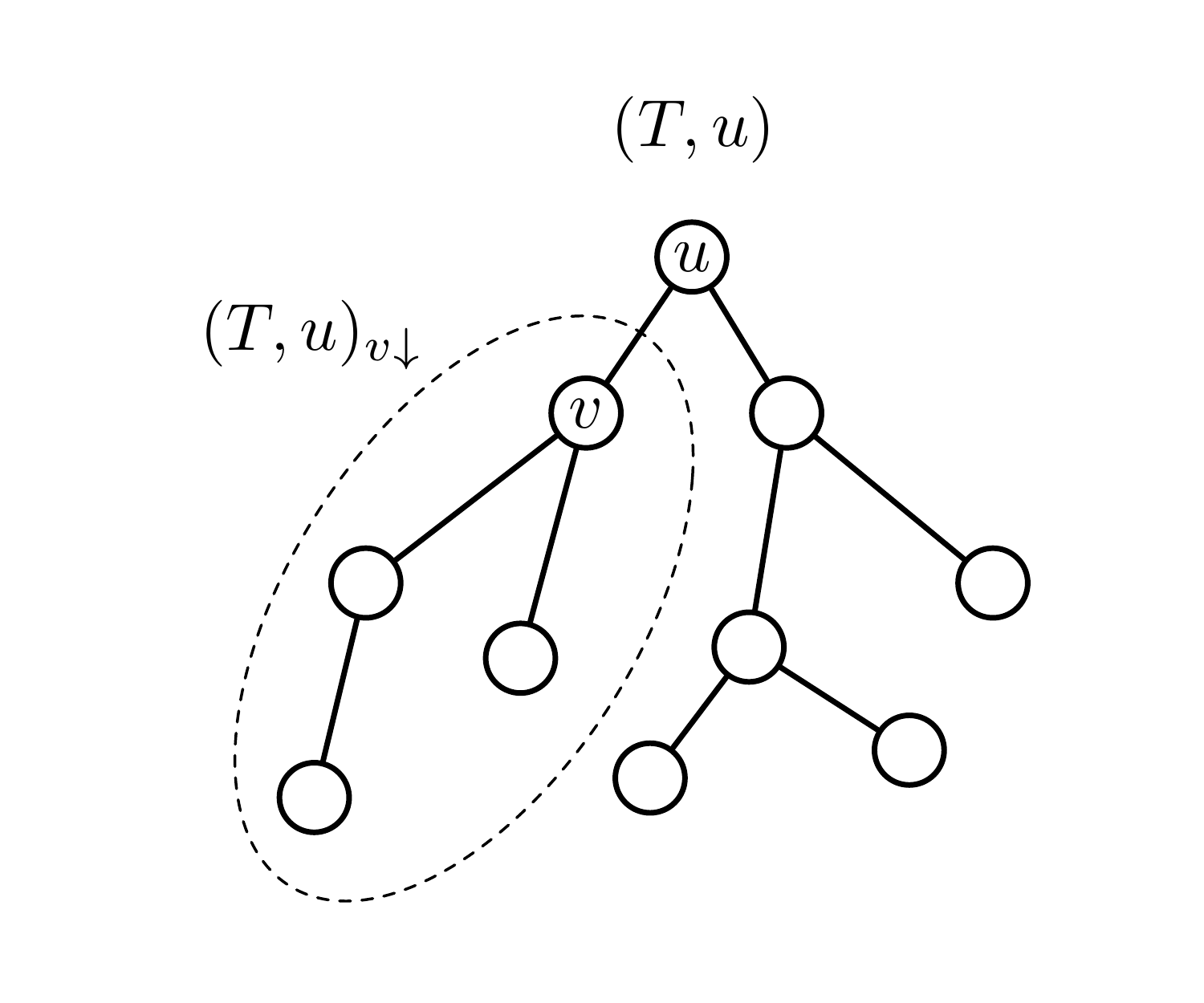}
\end{center}
\caption{A tree $T$ rooted at vertex $u$. The subtree $(T,u)_{v \downarrow}$ is highlighted.}
\label{FigDownarrow}
\end{figure}

\begin{theorem}
\label{ThmPersist}
Sublinear preferential attachment trees have a unique terminal centroid with probability 1.
\end{theorem}


The statement and proof of Theorem~\ref{ThmPersist} may be compared to the results obtained in our previous work~\cite{JogLoh15}, which establish persistent centrality for the special cases $\alpha = 0$ and $\alpha = 1$. For a subtree $T$, define the attractiveness of $T$ as the sum of the attraction functions evaluated at each vertex of $T$. In the case of uniform attachment, the attractiveness of $T$ is simply $|T|$, whereas for linear preferential attachment, it is the sum of the degrees of the vertices, which is $2|T|-1$. The linearity of attractiveness in $|T|$ was critical to obtaining sharp bounds on the diagonal crossing probability of certain random walks. When $\alpha \in (0,1)$, however, the attractiveness of $T$ is no longer a function of $|T|$ alone, rendering the methods of our previous work defunct. In the present paper, we leverage a continuous time embedding and convergence results for CMJ processes to prove terminal centrality for a large class of sublinear preferential attachment trees, with the tradeoff being a slightly weaker theoretical guarantee.

\begin{proof} [Proof of Theorem~\ref{ThmPersist} (sketch)]

The key steps of the proof are as follows:
\begin{itemize}
\item[(i)] Identify a necessary condition that a vertex must satisfy in order to be a terminal centroid.
\item[(ii)] Show that the set of vertices satisfying the condition in (i), called the set of \emph{candidate terminal centroids} and denoted by $\cC_{\text{CAN}}$, is nonempty and finite with probability 1.
\item[(iii)] Show that among the set of candidate terminal centroids, a unique vertex emerges that eventually becomes more central than any other candidate.
\item[(iv)] Show that the vertex in (iii) is the unique terminal centroid.
\end{itemize}

We first describe the necessary condition in step (i). (For an illustration, see Figure~\ref{FigVstar}.) Let $v^*(n)$ be a centroid of the tree $T_n$. If $T_n$ has two centroids, we choose $v^*(n)$ to be the younger vertex from among the two. If vertex $v_{n+1}$ is a terminal centroid, it must necessarily become more central than $v^*(n)$ after a finite amount of time. Consequently, let
\begin{equation*}
\cC_{\text{CAN}} := \{v_{n+1}: \exists M \text{ s.t. } \psi_{T_m}(v_{n+1}) < \psi_{T_m}(v^*(n)) \quad \forall m \ge M\},
\end{equation*}
and define $\cE_n$ to be the event $\{v_{n+1} \in \cC_{\text{CAN}}\}$. We follow the convention of considering $v_1$ to be a candidate terminal centroid; in particular, $\cC_{\text{CAN}} \neq \phi$.

In fact, for $n >1$, either $v^*(n)$ or $v_{n+1}$ eventually becomes more central than the other, which follows from the following lemma:
\begin{lemma*}\label{lemma: one of them wins}	
For any two vertices $u$ and $v$, there exists a time $M$ such that either $\psi_{T_m}(u) < \psi_{T_m}(v)$ or $\psi_{T_m}(u) > \psi_{T_m}(v)$ holds for all $m > M$, almost surely.
\end{lemma*}
\begin{proof}
Without loss of generality, assume $u$ is born before $v$. Let $T^v$ and $T^u$ denote the trees $(T_m, u)_{v \downarrow}$ and $(T_m, v)_{u \downarrow}$, where $m$ is the time of birth of $v$. Note that $T^v$ consists of the single vertex $v$. We now restart the process in continuous time; i.e., we start independent CMJ processes initiated from the starting states $T^u$ and $T^v$. Using Theorem \ref{thm: ABC}, we have the a.s. convergence result
\begin{equation}\label{eq: tu by tv}
\frac{|(T_m, v)_{u \downarrow}|}{|(T_m, u)_{v \downarrow}|} \stackrel{a.s.} \longrightarrow \frac{W^u}{W^v},
\end{equation}
for absolutely or singular continuous independent random variables $W^u$ and $W^v$, whose distributions are determined by the structure of the starting states $T^u$ and $T^v$, respectively. Since $W^u - W^v$ cannot have point masses, we have
\begin{equation*}
\mathbb P(W^u = W^v) = \mathbb P \left(W^u - W^v= 0 \right) = 0.
\end{equation*}
Thus, either $W^u > W^v$ or $W^u < W^v$, almost surely. The almost sure convergence in equation \eqref{eq: tu by tv}  implies that there exists $M>0$ such that either $|(T_m, v)_{u \downarrow}| > |(T_m, u)_{v \downarrow}|$ or  $|(T_m, v)_{u \downarrow}| < |(T_m, u)_{v \downarrow}|$, for all $m > M$. Applying Lemma~\ref{lemma: jogloh2.1} in Appendix~\ref{AppTrees} concludes the proof.
\end{proof}


\begin{figure}
\begin{center}
\includegraphics[scale = 0.5]{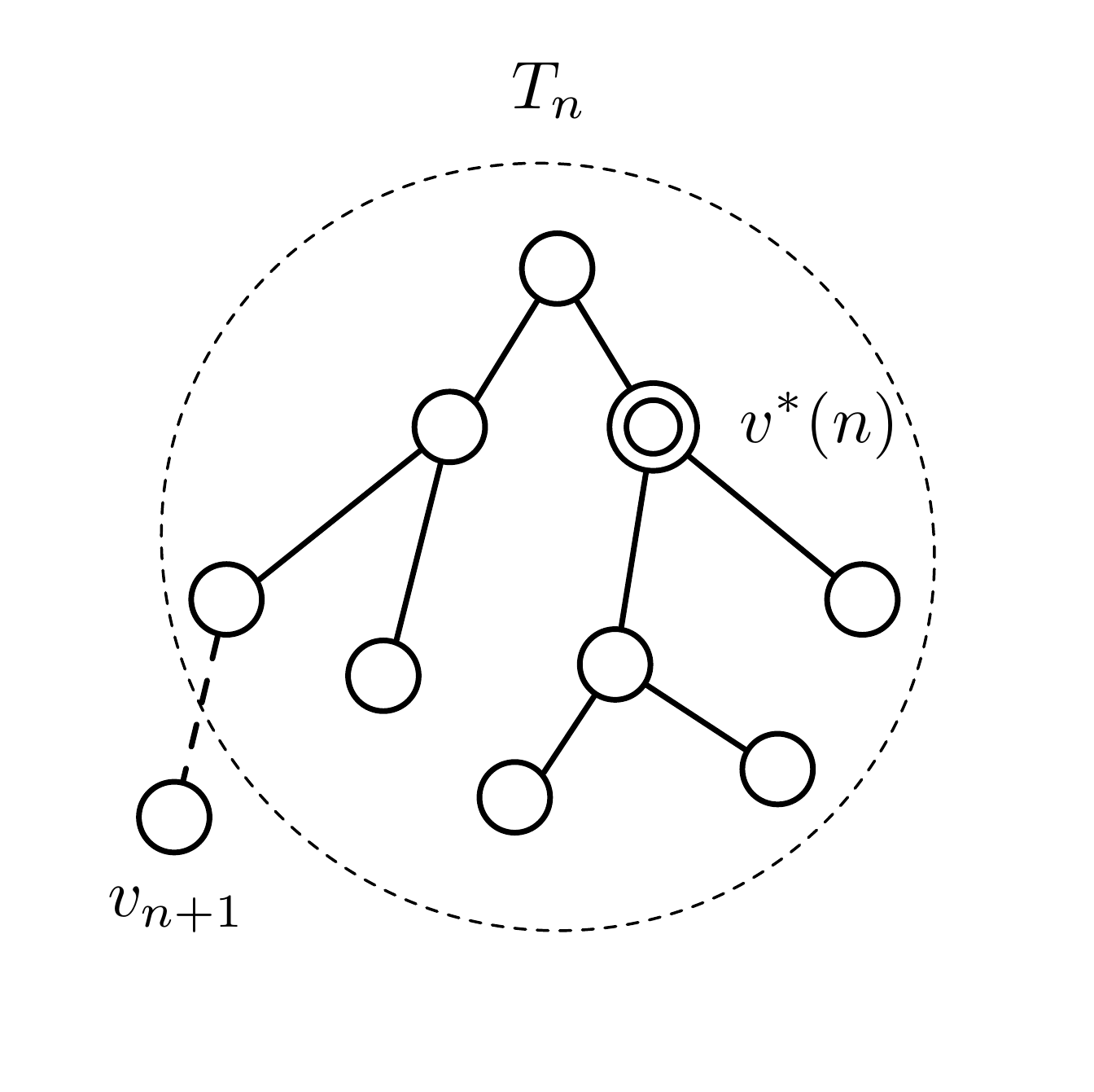}
\end{center}
\caption{The notation from Lemma \ref{lemma: finite} is illustrated above. The centroid of tree  $T_n$ is $v^*(n)$, and $v_{n+1}$ is the newest vertex joining $T_n$ to form $T_{n+1}$. }
\label{FigVstar}
\end{figure}

The following lemma furnishes the result in step (ii):

\begin{lemma*}\label{lemma: finite}
$|\cC_\text{CAN}| < \infty$, with probability 1.
\end{lemma*}
 
 \begin{proof}
We first show that any node joining the tree sufficiently late has a very small chance of belonging to $\cC_{\text{CAN}}$. By Lemma~\ref{lemma: jogloh2.1} in Appendix~\ref{AppTrees}, the event $\cE_n$ occurs if and only if there exists $M > 0$ such that for all $m \ge M$, 
\begin{equation}\label{eq: jogloh2.1}
|(T_m, v^*(n))_{v_{n+1}\downarrow}| > |(T_m, v_{n+1})_{v^*(n)\downarrow}|.
\end{equation}
To simplify notation, we define $A_m := (T_{m}, v^*(n))_{v_{n+1}\downarrow}$ and $B_m := (T_m, v_{n+1})_{v^*(n)\downarrow}$, for $m \geq n+1$. Lemma~\ref{lemma: jogloh3} in Appendix~\ref{AppTrees} implies that at time $m=n+1$, the number of vertices in $B_m$ is at least $\frac{n}{2}$. Thus, $B_m$ has a large lead over $A_m$, which has only one vertex. At time $n+1$, we pause the process in discrete time and restart it in continuous time, with state at $t=0$ being the state at the (discrete) time $n+1$. Observe that if a time $M$ exists such that inequality \eqref{eq: jogloh2.1} holds, a time $\Gamma > 0$ must also exist such that the continuous time trees satisfy $|A_\tau| > |B_\tau|$, for all $t > \Gamma$. 

Note that the population $|A_t|$ is simply a sublinear preferential attachment process started from a single vertex, which we denote by $Y(t)$. The population $|B_t|$ stochastically dominates the sum of $\frac{n}{2}$ independent sublinear preferential attachment processes starting from a single vertex, which we subsequently denote by $X_1(t), \dots, X_{n/2}(t)$. Thus, the probability that $\cE_n$ occurs is upper-bounded by the probability that $Y(t)$ eventually becomes larger than $\sum_{i=1}^{n/2} X_i(t)$. By Theorem \ref{thm: ABC}, the rescaled processes $e^{-\theta t}Y(t)$ and $e^{-\theta t} X_i(t)$, for $1 \leq i \leq n/2$, all converge a.s.\ to i.i.d.\ random variables, which we denote by $W_Y$ and $\{W_i\}_{1 \le i \le n/2}$, respectively. Thus, the probability that $Y(t)$ eventually becomes larger than $\sum_{i=1}^{n/2} X_i(t)$ is equal to the probability that $W_Y$ is greater than $\sum_{i=1}^{n/2} W_i$. Using Lemma~\ref{lemma: poling} in Appendix~\ref{appendix: wall}, we conclude that this probability  is upper-bounded by $\frac{C}{n^2}$, for some constant $C$. Finally, since $\sum \frac{1}{n^2}$ is a convergent sequence, the Borel-Cantelli lemma implies that with probability 1, only finitely many events $\cE_n$ occur, completing the proof.
 \end{proof}

For step (iii), we simply note that Lemma \ref{lemma: one of them wins} implies a fixed ordering via centrality for any two vertices. Thus, if we have a finite set such as $\cC_{\text{CAN}}$, a repeated application of Lemma \ref{lemma: one of them wins} to members of this set  yields a fixed ordering from the most central to the least central vertices in $\cC_\text{CAN}$. Let $v^*$ be the most central vertex from the set $\cC_{\text{CAN}}$ that emerges from this ordering. Step (iv) is provided by the following lemma:

\begin{lemma*}\label{lemma: frog}
The vertex $v^*$ is the unique terminal centroid.
\end{lemma*}
\begin{proof}
Let $u_0 \neq v^*$ be any vertex. If $u_0 \in \cC_{\text{CAN}}$, the choice of $v^*$ implies that $v^*$ eventually becomes more central than $u_0$. Thus, we assume $u_0 \notin \cC_{\text{CAN}}$, meaning the centroid at the time vertex $u_0$ was born, which we denote by $u_1$, eventually becomes more central than $u_0$ in the limit. If $u_1 \in \cC_{\text{CAN}}$, then $v^*$ eventually becomes more central than $u_1$, which in turn eventually becomes more central than $u_0$, as wanted. If instead $u_1 \notin \cC_{\text{CAN}}$, we may consider $u_2$, which is the centroid when $u_1$ was born. Continuing in this manner, we define a sequence $u_0, u_1, u_2, \dots$ of progressively older, which is necessarily finite, with the last vertex in the sequence being $v_1$. 
Thus, if we define
\begin{equation*}
r = \min_{i \geq 0}\left\{ u_{i} \in \cC_{\text{CAN}} \right\},
\end{equation*}
then $u_r$ is well-defined. We then have that $v^*$ is more central than $u_r$, which is more central than $u_{r-1}$, which is more central than $u_{r-2}$, and so on, continuing up to $u_0$. This completes the proof.
\end{proof}
This also completes the proof of Theorem \ref{ThmPersist}.
\end{proof}

In fact, Theorem~\ref{ThmPersist} may be extended to establish the existence of a fixed set of size $K>0$ consisting of the most terminally central vertices. This is summarized in the following theorem:
\begin{theorem}
\label{ThmTopK}
For any $K \ge 1$, a unique set of distinct vertices $\{v^*_1, v^*_2, \dots, v^*_K\}$ exists such that for any other vertex $u \in \cup_{n=1}^\infty V(T_n)$, there exists a time $M$ (possible dependent on $u$) such that
\begin{align*}
\psi_{T_n}(v^*_1) < \psi_{T_n}(v^*_2) < \dots < \psi_{T_n}(v^*_K) < \psi_{T_n}(u),
\end{align*}
for all $n \geq M$.
\end{theorem}

\begin{proof}
The argument closely parallels that of the proof of Theorem 2 in our previous work~\cite{JogLoh15}, with appropriate modifications to prove terminal centrality instead of persistent centrality. We refer the reader to our earlier paper, noting that the argument only requires properties of absolute or singular continuity of the appropriately normalized subtree sizes, which are provided by Theorem~\ref{thm: ABC}.
\end{proof}

\section{Finite confidence set for the root}\label{section: adam}

For the results in this section, we limit our consideration to $\alpha$-sublinear preferential attachment trees. Recall that these are trees in which the attraction function is given by $f(i) = (1+i)^\alpha$, for $\alpha \in (0,1)$. The problem of finding a confidence set for the root node in the case of linear preferential and uniform attachment trees was studied by Bubeck et al.~\cite{Bubeck14}. One proposed method for constructing a confidence set that contains the root node with probability $1-\epsilon$ is as follows:
\begin{enumerate}
\item Given a sequence of random trees $\{T_n\}$, order the vertices according to the balancedness function $\psi_{T_n}$.
\item Select the $K$ vertices with the smallest values of $\psi_{T_n}$, for a proper value of $K = K(\epsilon)$.
\end{enumerate}
The above method was shown to produce finite-sized confidence sets in Bubeck et al.~\cite{Bubeck14}, and the analysis was later extended to diffusions over regular trees~\cite{KhimLoh15}. In fact, the continuous time analysis of sublinear preferential attachment trees also furnishes a method for bounding the required size of a confidence set for the root node.
Following the notation of Bubeck et al.~\cite{Bubeck14}, we use $H^K_\psi(T_n)$ to denote the set of $K$ vertices chosen according to the method described above, and drop the argument $T_n$ when the context is unambiguous. Our main result shows that the same estimator produces finite-sized confidence sets for sublinear preferential attachment trees:

\begin{theorem}\label{thm: adam}
For $\epsilon >0$, there exists a constant $K$ (depending on $\epsilon$) such that
\begin{equation*}
\lim\inf_{n\to \infty} \mathbb P\left(v_1 \in H^K_\psi(T_n)\right) \geq 1- \epsilon.
\end{equation*}
\end{theorem}
\begin{figure}
\label{fig: tnik}
\begin{center}
\includegraphics[scale = 0.5]{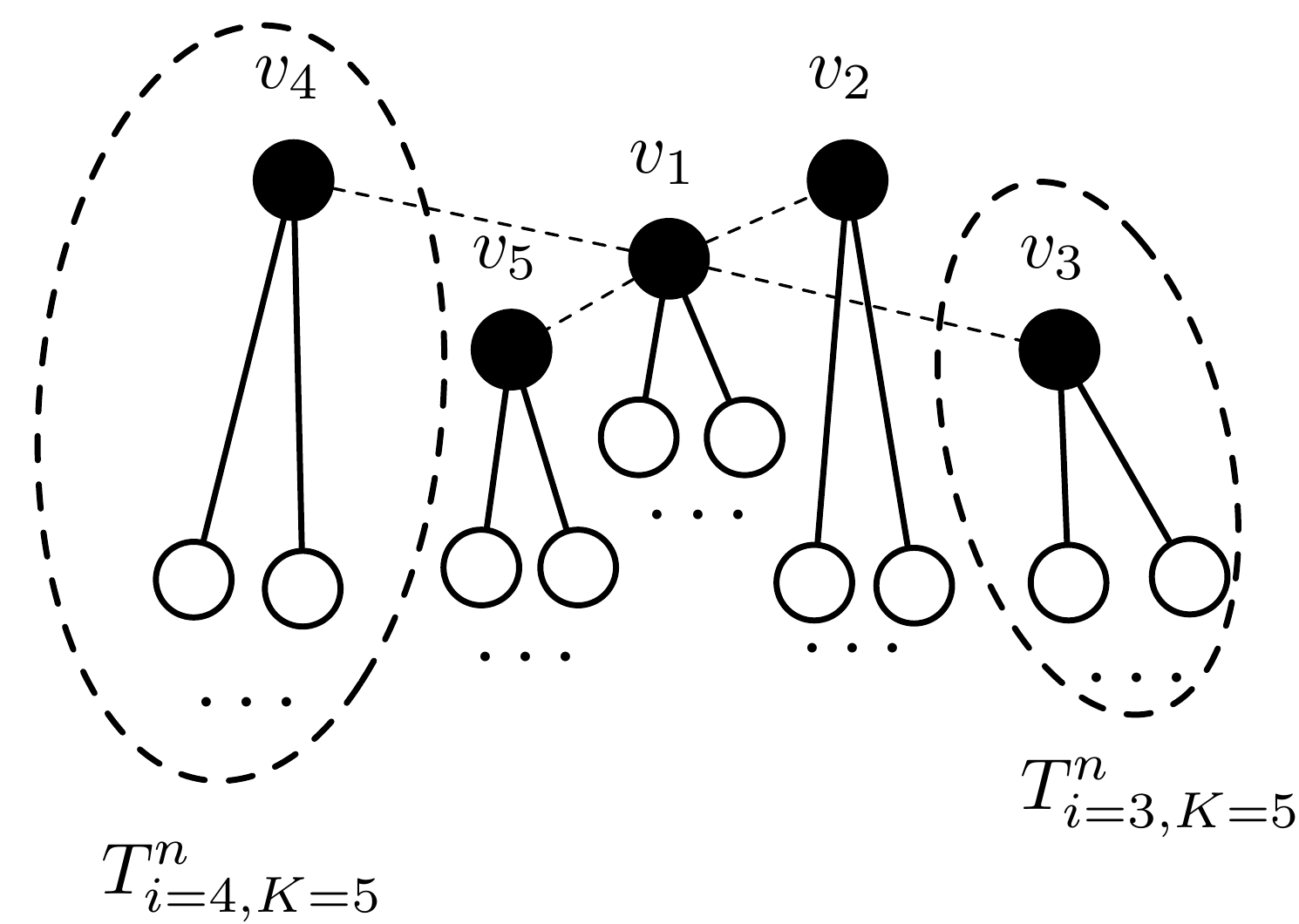}
\end{center}
\caption{An illustration of the trees $T^n_{i,K}$ defined in the proof of Theorem \ref{thm: adam}. The figure shows a tree with $K=5$, with the two trees $T^n_{3,5}$ and $T^n_{4,5}$ highlighted.}
\label{FigTnik}
\end{figure}

\begin{proof} [Proof of Theorem~\ref{thm: adam} (sketch)]
We follow the approach of Bubeck et al.~\cite{Bubeck14}. For $1 \le i \le K$, let $T^n_{i,K}$ denote the tree containing vertex $v_i$ in the forest obtained from $T_n$ by removing all edges between nodes $\{v_1, \dots, v_K\}$. (See Figure~\ref{FigTnik} for an illustration.) Observe that
\begin{align}
\label{EqnBlizzard}
\mathbb P(v_1 \notin H^K_\psi) & \leq \mathbb P\left(\exists i > K : \psi(v_i) \leq \psi(v_1)\right) \notag \\
& \leq \mathbb P(\psi(v_1) \geq (1-\delta)n) + \mathbb P(\exists i > K : \psi(v_i) \leq (1-\delta)n),
\end{align}
for $\delta > 0$ to be chosen later. To handle the first term in inequality~\eqref{EqnBlizzard}, we have the following lemma:
\begin{lemma*}[Proof in Appendix \ref{appendix: LemBhalu}]
\label{LemBhalu}
There exists $\delta_0 > 0$ such that
\begin{equation*}
\lim\sup_{n \to \infty} \mprob\left(\psi(v_1) \geq (1-\delta_0)n\right) < \frac{\epsilon}{2}.
\end{equation*}
\end{lemma*}
The proof of the above lemma is simple and follows by an argument similar to that in Bubeck et al.~\cite{Bubeck14}. The analysis of the second term in inequality~\eqref{EqnBlizzard} is more technical:

\begin{lemma*}[Proof in Appendix \ref{appendix: LemHouston}]
\label{LemHouston}
There exist constants $N$ and $C$ depending only on $\epsilon$ such that if $K>N$ and
$$\frac{CK(\log K)^{\frac{2}{1-\alpha}}}{(K-1)^2} < \frac{\epsilon}{4},$$
then
\begin{equation*}
\lim\sup_{n \to \infty} \mathbb P\left(\exists i > K : \psi(v_i) \leq (1-\delta_0)n\right) < \frac{\epsilon}{2}.
\end{equation*}
\end{lemma*}

A brief proof sketch of Lemma~\ref{LemHouston} is as follows. First, we claim that for any $i > K$, 
$$\psi(v_i) \geq \min_{1 \leq k \leq K} \sum_{j = 1, j \neq k}^K |T^n_{j,K}|.$$
This is because $v_i$ must lie in one of the trees $T^n_{k, K}$, for some $1 \leq k \leq K$. Thus, the largest subtree hanging off $v_i$ is at least as large as the subtree of $(T, v_i)$ containing vertex $v_k$. From Figure \ref{fig: tnik}, we see that this is at least $\sum_{j=1, j \neq k}^K |T^n_{j, K}|$, which is in turn at least $\min_{1 \leq k \leq K} \sum_{j = 1, j \neq k}^K |T^n_{j,K}|$.
Hence, the desired expression may be lower-bounded by
\begin{multline*}
\mathbb P\Big(\exists 1 \leq k \leq K : \sum_{j=1, j\neq k}^K |T^n_{j,K}| \leq (1-\delta_0)n\Big)
\stackrel{(a)}= \mathbb P\Big(\exists 1 \leq k \leq K : \sum_{j=1, j\neq k}^K |T^n_{j,K}| \leq \frac{1-\delta_0}{\delta_0} |T_{k, K}| \Big),
\end{multline*}
where $(a)$ follows because $\sum_{j=1}^K |T^n_{j, K}|$ is simply the total number of vertices, which is $n$.
This final term may be bounded from above by observing that (i) the growth rate of $|T^n_{k,K}|$ is larger for a large degree of $v_k$ in $T_K$; and (ii) the maximum degree of $T_K$ is roughly $(\log K)^{1/(1-\alpha)}$, which is not sufficient to increase the growth of $|T^n_{k,K}|$ so as to compete with the sum of $K-1$ trees given by $\sum_{j=1, j\neq k}^K |T^n_{j,K}|$, when $K$ is sufficiently large. \\

Combining the bounds of Lemmas~\ref{LemBhalu} and~\ref{LemHouston} and substituting back into inequality~\eqref{EqnBlizzard} completes the proof.
\end{proof}



\section{Discussion} \label{section: discussion}

In this paper, we have established the existence of a unique terminal centroid in sublinear preferential attachment trees. However, our results have stopped short of proving that a persistent centroid exists, which was conjectured in our previous work~\cite{JogLoh15}. To establish the stronger statement, it would suffice to show that the terminal centroid identified in our paper is in fact persistent. (Although somewhat counterintuitive, the definition of terminal centrality leaves open the possibility that the terminal centroid never actually becomes the tree centroid at any finite time point.)

A possible approach leverages ideas from our previous work~\cite{JogLoh15}. We now describe the main bottleneck in extending the argument employed there in the present setting. For the purpose of this discussion, suppose the sublinear preferential attachment tree has an attraction function $f(i) = (i+1)^\alpha$, for some $0 < \alpha < 1$. The analog of our previous approach~\cite{JogLoh15} would involve two steps: (i) showing that the total number of vertices that ever become centroids is finite, almost surely; and (ii) concluding the existence of unique persistent centroid by an application of Lemma \ref{lemma: one of them wins}. In the first step, we need to show that vertices which are born late have a very small probability of ever becoming the tree centroid at any future point in time, and then apply the Borel-Cantelli lemma to establish finiteness. As described in the proof of Theorem \ref{ThmPersist}, a vertex $v_{n+1}$ can become centroid at some point in the future if and only if the tree $A_t$ is able to catch up with the tree $B_t$ at some future time $t$. However, unlike in the case of uniform or linear preferential attachment trees, the probability that a new vertex joins $A_t$ or $B_t$ is no longer proportional to a simple linear function of the size of the subtree. Based on these ideas, however, one can show that a sufficient condition for persistent centrality in the tree growth process is as follows:

\paragraph{\textbf{Irrelevance of structure condition:}}
There exists a parameter $\eta \in (0,1)$ such that for any two trees $\Gamma_1$ and $\Gamma_2$ with $|\Gamma_1| = |\Gamma_2|$, the probability that the CMJ process started from $\Gamma_1$ has a larger population than the CMJ process started from $\Gamma_2$, in the limit, lies in the interval $(\eta, 1-\eta)$. \\

The above condition essentially ensures that the structure of the tree does not have a significant impact on how quickly it grows.

Note that for linear preferential attachment and uniform attachment trees, we can take any $\eta < 1/2$, since the probability that the population of $\Gamma_1$ is larger than the population of $\Gamma_2$ in the limit is exactly $\frac{1}{2}$. Irrelevance of structure is thus crucially leveraged in the proof strategy of our previous work~\cite{JogLoh15}. Unfortunately, the irrelevance of structure condition does not appear to hold for sublinear preferential attachment trees, due to the following small example:
\begin{example}
Let $\Gamma_1$ be the ``line tree"; i.e. a sequence of $r$ vertices $\{v_1, v_2, \dots, v_r\}$ such that the edge set is $\{(v_i, v_{i+1}): 1 \leq i \leq r-1\}$. Let $\Gamma_2$ be the ``star tree" with the edge set being $\{(v_1, v_i): 2 \leq i \leq r \}$. If the population of the CMJ process started from $\Gamma_i$ is asymptotically $W_i e^{\theta t}$ for $i \in \{1,2\}$, it is possible to show that $\frac{W_1}{r}$ and $\frac{W_2}{r}$ converge in probability to some constants $c_1$ and $c_2$, such that $c_1 > c_2$, as $r \rightarrow \infty$. Thus, the probability that the population of $\Gamma_1$ is larger than the population of $\Gamma_2$ in the limit tends to 1 as $r \rightarrow \infty$. This violates the irrelevance of structure condition, since no matter which $\eta > 0$ is chosen, the probability of $\Gamma_1$ being larger than $\Gamma_2$ in the limit surpasses $1-\eta$ for all large enough $r$.
\end{example}

Of course, line trees and star trees do not typically show up in sublinear preferential attachment trees, so it may be possible to redefine the irrelevance of structure property to rule out such low-probability configurations. However, a suitable modification that paves the way to proving persistent centrality has yet to be determined.

Finally, suppose we define $\cC_\infty$ to be the set of all vertices which are tree centroids for infinite amounts of time. It is easy to see that Lemma \ref{lemma: one of them wins} rules out the possibility of $\cC_\infty \geq 2$, since any two vertices in $\cC_\infty$ will have a fixed centrality ordering in the limit. Note that $\cC_\infty = 0$ implies that an infinite number of vertices ever become centroids, albeit for finite amounts of time; whereas $\cC_\infty = 1$ also does not preclude such a possibility. A weaker conjecture than persistent centrality, but stronger than terminal centrality, would therefore be to show that $|\cC_\infty| = 1$. This question currently remains open.

Regarding root inference, it is an open question whether the method of constructing finite confidence sets based on centrality continues to hold beyond the subclass of $\alpha$-sublinear preferential attachment trees. Also note that the results on root inference in this paper are generally weaker than those obtained for linear preferential and uniform attachment trees in Bubeck et al.~\cite{Bubeck14} and Jog and Loh~\cite{JogLoh15}, since we have not provided bounds on the size of a confidence set for the root node, or the size of the hub around $v_1$ that will ensure its persistent centrality. The main hurdle in establishing such bounds is, again, the lack of concrete information about the limiting random variable $W$ in a CMJ process. Although obtaining the exact distribution of $W$ seems too optimistic, it may be possible to obtain bounds on moments or tail probabilities, which could be used to obtain bounds on hub sizes or confidence sets. Our results in this paper also strengthen the belief that the age of a node and its centrality are strongly related in growing random trees, implying that it is extremely difficult for a vertex to hide its age. Fanti et al.~\cite{FanEtal15} explored the problem of how to create a diffusion process over a regular tree in order to obfuscate the oldest node, and it would be very interesting to see if classes of attraction functions exist that cause the tree to grow in such a way that the best confidence set for the root node does not remain finite as the tree grows.


\section*{Acknowledgment}
The authors would like to thank the AE and three anonymous reviewers for their helpful and positive feedback while preparing the revision.

\bibliography{refs}

\begin{thebibliography}{10}

\bibitem{AlbBar02}
R.~Albert and A.-L. Barab{\'a}si.
\newblock Statistical mechanics of complex networks.
\newblock {\em Rev. Mod. Phys.}, 74:47--97, Jan 2002.

\bibitem{AthNey04}
K.~B. Athreya and P.~Ney.
\newblock {\em Branching Processes}.
\newblock Dover Books on Mathematics. Dover Publications, 2004.

\bibitem{Bar16}
A.-L. Barab{\'a}si.
\newblock {\em Network Science}.
\newblock Cambridge University Press, 2016.

\bibitem{Barabasi99}
A.-L. Barab{\'a}si and R.~Albert.
\newblock Emergence of scaling in random networks.
\newblock {\em Science}, 286(5439):509--512, 1999.

\bibitem{BarEtal08}
A.~Barrat, M.~Barthlemy, and A.~Vespignani.
\newblock {\em Dynamical Processes on Complex Networks}.
\newblock Cambridge University Press, New York, NY, USA, 2008.

\bibitem{Bhamidi07}
S.~Bhamidi.
\newblock Universal techniques to analyze preferential attachment trees:
  {G}lobal and local analysis.
\newblock {\em In preparation}, August 2007.

\bibitem{Biggins79}
J.~D. Biggins and D.~R. Grey.
\newblock Continuity of limit random variables in the branching random walk.
\newblock {\em Journal of Applied Probability}, pages 740--749, 1979.

\bibitem{Bubeck14}
S.~Bubeck, L.~Devroye, and G.~Lugosi.
\newblock Finding {A}dam in random growing trees.
\newblock {\em arXiv preprint arXiv:1411.3317}, 2014.

\bibitem{CrumpMode68}
K.~S. Crump and C.~J. Mode.
\newblock A general age-dependent branching process. {I}.
\newblock {\em Journal of Mathematical Analysis and Applications},
  24(3):494--508, 1968.

\bibitem{CrumpMode69}
K.~S. Crump and C.~J. Mode.
\newblock A general age-dependent branching process. {II}.
\newblock {\em Journal of Mathematical Analysis and Applications}, 25(1):8--17,
  1969.

\bibitem{DerMoe09}
S.~Dereich and P.~M\"{o}rters.
\newblock Random networks with sublinear preferential attachment: {D}egree
  evolutions.
\newblock {\em Electron. J. Probab.}, 14(43):1222--1267, 2009.

\bibitem{Doney72}
R.~A. Doney.
\newblock A limit theorem for a class of supercritical branching processes.
\newblock {\em Journal of Applied Probability}, pages 707--724, 1972.

\bibitem{Dong13}
W.~Dong, W.~Zhang, and C.-W. Tan.
\newblock Rooting out the rumor culprit from suspects.
\newblock In {\em Information Theory Proceedings (ISIT), 2013 IEEE
  International Symposium on}, pages 2671--2675. IEEE, 2013.

\bibitem{DorMen03}
S.~N. Dorogovtsev and J.~F.~F. Mendes.
\newblock {\em Evolution of Networks: From Biological Nets to the Internet and
  WWW}.
\newblock Oxford University Press, Inc., New York, NY, USA, 2003.

\bibitem{EckMoe14}
M.~Eckhoff and P.~M\"{o}rters.
\newblock Vulnerability of robust preferential attachment networks.
\newblock {\em Electron. J. Probab.}, 19(57):1--47, 2014.

\bibitem{FanEtal15}
G.~{Fanti}, P.~{Kairouz}, S.~{Oh}, K.~{Ramchandran}, and P.~{Viswanath}.
\newblock {Hiding the Rumor Source}.
\newblock {\em ArXiv e-prints}, September 2015.

\bibitem{Galashin13}
P.~Galashin.
\newblock Existence of a persistent hub in the convex preferential attachment
  model.
\newblock {\em arXiv preprint arXiv:1310.7513}, 2013.

\bibitem{Har12}
T.~E. Harris.
\newblock {\em The Theory of Branching Processes}.
\newblock Grundlehren der mathematischen Wissenschaften. Springer Berlin
  Heidelberg, 2012.

\bibitem{Jack08}
M.~O. Jackson.
\newblock {\em Social and Economic Networks}, volume~3.
\newblock Princeton University Press, 2008.

\bibitem{Jagers75}
P.~Jagers.
\newblock {\em Branching processes with biological applications}.
\newblock Wiley, 1975.

\bibitem{JagerNerman84}
P.~Jagers and O.~Nerman.
\newblock The growth and composition of branching populations.
\newblock {\em Advances in Applied Probability}, pages 221--259, 1984.

\bibitem{JogLoh15}
V.~Jog and P.~Loh.
\newblock Persistence of centrality in random growing trees.
\newblock {\em arXiv preprint arXiv:1511.01975}, 2015.

\bibitem{Jor1869}
C.~Jordan.
\newblock Sur les assemblages de lignes.
\newblock {\em J. Reine Angew. Math}, 70(185):81, 1869.

\bibitem{Karam13}
N.~Karamchandani and M.~Franceschetti.
\newblock Rumor source detection under probabilistic sampling.
\newblock In {\em Information Theory Proceedings (ISIT), 2013 IEEE
  International Symposium on}, pages 2184--2188. IEEE, 2013.

\bibitem{KhimLoh15}
J.~Khim and P.~Loh.
\newblock Confidence sets for the source of a diffusion in regular trees.
\newblock {\em arXiv preprint arXiv:1510.05461}, 2015.

\bibitem{KraEtal00}
P.~L. {Krapivsky}, S.~{Redner}, and F.~{Leyvraz}.
\newblock Connectivity of growing random networks.
\newblock {\em Physical Review Letters}, 85:4629--4632, November 2000.

\bibitem{KunEtal13}
J.~Kunegis, M.~Blattner, and C.~Moser.
\newblock Preferential attachment in online networks: {M}easurement and
  explanations.
\newblock In {\em Proceedings of the 5th Annual ACM Web Science Conference},
  WebSci '13, pages 205--214, New York, NY, USA, 2013. ACM.

\bibitem{LuoEtal13}
W.~Luo, W.-P. Tay, and M.~Leng.
\newblock Identifying infection sources and regions in large networks.
\newblock {\em {IEEE} Trans. Signal Processing}, 61(11):2850--2865, 2013.

\bibitem{Mit78}
S.~L. Mitchell.
\newblock Another characterization of the centroid of a tree.
\newblock {\em Discrete Mathematics}, 24(3):277--280, 1978.

\bibitem{Nerman81}
O.~Nerman.
\newblock On the convergence of supercritical general {(CMJ)} branching
  processes.
\newblock {\em Probability Theory and Related Fields}, 57(3):365--395, 1981.

\bibitem{New10}
M.~Newman.
\newblock {\em Networks: An Introduction}.
\newblock Oxford University Press, Inc., New York, NY, USA, 2010.

\bibitem{Rudas07}
A.~Rudas, B.~T{\'o}th, and B.~Valk{\'o}.
\newblock Random trees and general branching processes.
\newblock {\em Random Structures \& Algorithms}, 31(2):186--202, 2007.

\bibitem{Shah11}
D.~Shah and T.~Zaman.
\newblock Rumors in a network: {W}ho's the culprit?
\newblock {\em IEEE Transactions on Information Theory}, 57(8):5163--5181,
  2011.

\bibitem{Shah15}
D.~Shah and T.~Zaman.
\newblock Finding rumor sources on random trees.
\newblock {\em arXiv preprint arXiv:1110.6230}, 2015.

\bibitem{Sla75}
P.~J. Slater.
\newblock Maximin facility location.
\newblock {\em Journal of National Bureau of Standards B}, 79:107--115, 1975.

\bibitem{Sla81}
P.~J. Slater.
\newblock Accretion centers: A generalization of branch weight centroids.
\newblock {\em Discrete Applied Mathematics}, 3(3):187--192, 1981.

\bibitem{TanEtAl16}
C.~W. Tan, P.-D. Yu, C.-K. Lai, W.~Zhang, and H.-L. Fu.
\newblock Optimal detection of influential spreaders in online social networks.
\newblock In {\em 2016 Annual Conference on Information Science and Systems
  (CISS)}, pages 145--150. IEEE, 2016.

\bibitem{Zel68}
B.~Zelinka.
\newblock Medians and peripherians of trees.
\newblock {\em Archivum Mathematicum}, 4(2):87--95, 1968.

\end{thebibliography}

\begin{appendix}
\section{Results on CMJ processes}\label{appendix: cmj}

In this Appendix, we review properties of CMJ processes and verify that the CMJ process corresponding to a sublinear preferential attachment tree enjoys certain convergence properties.

\subsection{Preliminary results}

We begin by stating several results that will be crucial for our purposes. For a more detailed discussion of such results, see the survey paper by Jager and Nerman~\cite{JagerNerman84}.

\begin{lemma*}[Corollary 4.2 and Theorem 4.3 from Jager and Nerman~\cite{JagerNerman84}]\label{lem: L2}
Let $\xi$ be a point process on $\mathbb R_+$ with Malthusian parameter $\theta > 0$. Consider a CMJ process driven by $\xi$ in which individuals live forever. Let the population of the CMJ process at time $t \geq 0$ be denoted by $Z_t$. Define
$$\hat \xi(\theta) = \int_0^\infty e^{-\theta t}d\xi(t).$$
If the condition
\begin{equation}
\label{EqnStar}
\qquad \text{Var}(\hat \xi(\theta)) < \infty \tag{$\star$}
\end{equation}
is satisfied, then we have the convergence result
\begin{equation*}
e^{-\theta t}Z_t \stackrel{L^2} \longrightarrow W,
\end{equation*}
where $W$ is a random variable satisfying $W > 0$, almost surely.
\end{lemma*}

\begin{lemma*}[Theorem 5.4 from Nerman~\cite{Nerman81}]\label{lem: ass}
Let $\xi, \theta,$ and $Z_t$ be as in Lemma~\ref{lem: L2}. If the mean intensity measure $\mu$ satisfies
\begin{equation}
\label{EqnStarStar}
\qquad \int_0^\infty e^{-\tilde\theta t}d\mu(t) < \infty, \quad \quad \text{for some} \quad \tilde\theta < \theta,
\tag{$\star\star$}
\end{equation}
then
\begin{equation*}
e^{-\theta t}Z_t \stackrel{a.s.} \longrightarrow W,
\end{equation*}
where $W$ is as in Lemma~\ref{lem: L2}.\footnote{In Nerman~\cite{Nerman81}, the condition~\eqref{EqnStarStar} appears in a more general form denoted Condition 5.1. As explained in the remark following Condition 5.1, the condition~\eqref{EqnStarStar} is stronger and implies Condition 5.1.}
\end{lemma*}

Although not much is known about the exact distribution of $W$ in the case of a general CMJ process, the following useful properties have been established:
\begin{lemma*}[Theorem 1 from Biggins and Grey~\cite{Biggins79}]\label{lem: sing}
Let $W$ be the limit random variable appearing in Lemmas~\ref{lem: L2} and~\ref{lem: ass}. The the following properties hold:
\begin{itemize}
\item[(i)] The distribution of $W$ has no atoms.
\item[(ii)] The distribution of $W$ is either singular continuous or absolutely continuous.
\item[(iii)] The support of $W$ is all of $\mathbb R_+$; i.e., the set of positive points of increase of the distribution function of $W$ is all of $\mathbb R_+$.
\end{itemize}
\end{lemma*}

\begin{remark*}
Note that in all the above results, we have assumed that the branching process begins with a single individual. Suppose, however, that the process starts from some initial state consisting of a finite collection of nodes $\{v_1, \dots, v_k\}$ satisfying parent-child relationships according to a directed tree $T$ rooted at $v_1$. In this case, we can condition the CMJ process beginning with a single node on the event of observing the tree $T$ at some point, and conclude that the Malthusian normalized population converges to a random variable $\tilde W$ almost surely and in $L^2$. Although we do not provide a proof here, the limit random variable $\tilde W$ also satisfies all the properties in Lemma~\ref{lem: sing}.
\end{remark*}

\subsection{Sublinear preferential attachment}

We now specialize our discussion to sublinear preferential attachment processes.

\begin{lemma*}\label{lemma: malthus}
The Malthusian parameter $\theta$ for a sublinear preferential process always exists and satisfies $1 < \theta < 2$.
\end{lemma*}

\begin{proof}
A stronger version of this lemma may be found in Lemma 44 of Bhamidi~\cite{Bhamidi07}. Let $\xi$ be the point process associated with a sublinear preferential attachment function $f$ with mean intensity $\mu(t)$. Let $\mu_{UA}(t)$ and $\mu_{PA}(t)$ be the mean intensities of the standard Yule process and the Poisson process with rate 1, respectively. Clearly, the mean intensity functions satisfy
\begin{equation*}
\mu_{UA}(t) < \mu(t) < \mu_{PA}(t).
\end{equation*}
Let $X_\theta$ be an exponential random variable with rate $\theta$, independent of $\xi$. Note that the integral
\begin{align*}
\theta \int_0^\infty e^{-\theta t}\mu(t) dt = \mathbb E[\xi(X_\theta)]
\end{align*}
is monotonically decreasing in $\theta$. At $\theta = 1$, using the fact that $\mu_{UA}(t) < \mu(t)$, we have $1 < \mathbb E[\xi(X_1)]$. Similarly, at $\theta =2$, we may use the fact that $\mu(t) < \mu_{PA}(t)$ to obtain $\mathbb E [\xi(X_2)] <1$. By monotonicity, the value of $\E[\xi(X_\theta)]$ must therefore equal 1 at some $1 < \theta < 2$.
\end{proof}

\begin{lemma*}\label{lemma: stars}
The point process $\xi$ corresponding to a sublinear attraction function $f$ satisfies conditions~\eqref{EqnStar} and~\eqref{EqnStarStar}.
\end{lemma*}

\begin{proof}
We first show that condition~\eqref{EqnStar} is satisfied by following an approach used in Bhamidi~\cite{Bhamidi07}. For $0< \alpha < 1$, let $1 \leq f(i) \leq (i+1)^\alpha$ be a sublinear attraction function and let $\xi$ be the associated point process with Malthusian parameter $\theta$, existing by Lemma~\ref{lemma: malthus}. Let $X_\theta$ be an exponential random variable with rate $\theta$, independent of $\xi$. Defining the random function
\begin{equation*}
\hat \xi(\theta) := \int_0^\infty e^{-\theta t} d\xi(t),
\end{equation*}
we have by Fubini's theorem that
\begin{equation*}
\hat \xi(\theta) = \theta \int_{0}^\infty e^{-\theta t}\xi(0, t]dt = \mathbb E \left[\xi(0, X_\theta] \mid \xi\right].
\end{equation*}
Then
\begin{align*}
\text{Var}(\hat \xi(\theta)) &\leq \mathbb E \left[\hat \xi(\theta)^2\right]\\
&= \mathbb E \left[ \left( \mathbb E \left[ \xi(0, X_\theta] \mid \xi \right] \right)^2 \right]\\
& \stackrel{(a)}{\le} \E\left[\E\left[\xi(0, X_\theta]^2 \mid \xi\right]\right] \\
&= \mathbb E \left[\xi(0, X_\theta]^2\right],
\end{align*}
where inequality $(a)$ follows from Jensen's inequality. Thus, it is enough to derive the bound \mbox{$\mathbb E \left[\xi(0, X_\theta]^2\right] < \infty$}. Let $\xi_{\alpha}$ be the the point process corresponding to the the attraction function $f_{\alpha}(i) = (1+i)^\alpha$. Note that since 
\begin{equation*}
\mathbb E\left[ \xi(0, X_\theta)^2\right] \leq \mathbb E \left[ \xi_{\alpha}(0, X_\theta]^2\right],
\end{equation*}
it is enough to show that 
$$E \left[ \xi_{\alpha}(0, X_\theta]^2\right] < \infty.$$
Note that it is possible to find the exact distribution of the random variable $\xi_{\alpha}(0, X_\theta]$, as follows: The time of the $k^\text{th}$ arrival in the point process $\xi_{\alpha}$ may be written as $\sum_{i=1}^k Y_i$, where $Y_i \sim Exp\left(f_{\alpha}(i-1)\right)$ and the $Y_i$'s are independent. Hence,
\begin{align*}
\mathbb P(\xi_{\alpha}(0, X_\theta] \geq k) &= \mathbb P\left(X_\theta \geq \sum_{i=1}^k Y_i\right) \\
&= \mathbb E\left[e^{-\theta\sum_{i=1}^k Y_i} \right]\\
&= \prod_{i=0}^{k-1} \frac{f_{\alpha}(i)}{\theta + f_{\alpha}(i)} \\
&= \prod_{i=0}^{k-1} \frac{(1+ i)^\alpha}{\theta + (1+ i)^\alpha}.
\end{align*}
The probability mass function of $\xi_{\alpha}(0, X_\theta]$ is thus given by
\begin{align*}
\mathbb P(\xi_{\alpha}(0, X_\theta] = k) &= \prod_{i=0}^{k-1} \frac{(1+ i)^\alpha}{\theta + (1 + i)^\alpha} - \prod_{i=0}^{k} \frac{(1+ i)^\alpha}{\theta + (1 + i)^\alpha}\\
&= \frac{\theta}{\theta + (1+k)^\alpha}\prod_{i=0}^{k-1} \frac{(1+ i)^\alpha}{\theta + (1 + i)^\alpha} \\
& \sim \frac{1}{k^\alpha}\exp \left( -\frac{\theta k^{1-\alpha}}{1-\alpha} \right).
\end{align*}
It is now easy to check that $\mathbb E \left[\xi_\alpha(0, X_\theta]^2\right] < \infty$, and thus, $\mathbb E \left[\xi(0, X_\theta]^2\right] < \infty$.

Finally, we show that condition~\eqref{EqnStarStar} holds. Let $\mu(t)$ be the intensity measure associated with the sublinear preferential attachment process. Let $\tilde \theta$ be such that $1 < \tilde \theta < \theta$. Note that such a parameter $\tilde \theta$ exists by Lemma \ref{lemma: malthus}. As in Lemma \ref{lemma: malthus}, let $\mu_{PA}$ be the mean intensity measure associated with the linear preferential attachment process. Then
\begin{align*}
\int_0^\infty e^{-\tilde\theta t}d\mu(t) &< \int_0^\infty e^{-\tilde\theta t}d\mu_{PA}(t)\\
&\stackrel{(a)}= \int_0^\infty e^{(1-\tilde\theta) t}dt < \infty,
\end{align*}
where equality~$(a)$ holds because $\mu_{PA}(t) = e^t-1$.
\end{proof}

\subsection{Proof of Theorem~\ref{thm: ABC}}

Having verified the conditions~\eqref{EqnStar} and~\eqref{EqnStarStar} via Lemma~\ref{lemma: stars}, we obtain the desired $L^2$ and almost sure convergence by applying Lemmas~\ref{lem: L2} and~\ref{lem: ass}, respectively. The absolute or singular continuity of the limit random variable follows from Lemma~\ref{lem: sing}.

\section{Useful results on trees}
\label{AppTrees}

In this Appendix, we collect three key lemmas concerning trees and tree centroids that we use in our proofs.

\begin{lemma*} [Lemma 2.1 from Jog and Loh~\cite{JogLoh15}] \label{lemma: jogloh1}
For a tree $T$ on $n$ vertices, the following statements hold:
\begin{itemize}
\item[(i)] If $v^*$ is a centroid, then
$$\psi_T(v^*) \leq \frac{n}{2}.$$
\item[(ii)] $T$ can have at most two centroids.
\item[(iii)] If $u^*$ and $v^*$ are two centroids, then $u^*$ and $v^*$ are adjacent vertices. Furthermore,
\begin{equation*}
\psi_T(u^*) = |(T, u^*)_{v^*\downarrow}|, \quad \text{and} \quad \psi_T(v^*) = |(T, v^*)_{u^*\downarrow}|.
\end{equation*}
\end{itemize}
\end{lemma*}

\begin{lemma*} [Lemma 2.3 from Jog and Loh~\cite{JogLoh15}]\label{lemma: jogloh3}
Let $\{T_n\}_{n \ge 1}$ be a sequence of growing trees, with $V(T_n) = \{v_1, \dots, v_n\}$. At time $n+1$, we have the inequality
\begin{align*}
|(T_{n+1}, v_{n+1})_{v^*(n)\downarrow}| \geq \frac{n}{2}.
\end{align*}
\end{lemma*}

\begin{lemma*} \label{lemma: jogloh2.1}
Consider a tree $T$ and pick any two vertices $u, v \in V(T)$. Then we have the following result:
\begin{equation*}
\psi_T(u) \leq \psi_T(v) \iff |(T,v)_{u \downarrow}| \geq |(T, u)_{v \downarrow}|.
\end{equation*}
\end{lemma*}

\begin{proof}
Let $u'$ and $v'$ be the neighboring vertices to $u$ and $v$, respectively, in the path from $u$ to $v$. To simplify notation, denote $|(T,v)_{u \downarrow}| = a$ and $|(T, u)_{v \downarrow}| = b$.

First suppose $a \geq b$. Let $c = |T| - a - b$ be number of vertices not in either of the two subtrees. (See Figure~\ref{FigABC}.) We have the following inequality:
\begin{equation*}
\psi_T(v) \geq |(T,v)_{v' \downarrow}| = a + c.
\end{equation*}
We also have the inequality 
\begin{align*}
\psi_T(u) &\leq \max\left( |(T,v)|_{u \downarrow}-1, \; |(T,u)_{u' \downarrow}|\right)\\
&= \max \left(a-1, \; b + c \right)\\
&\stackrel{(a)} \leq a + c,
\end{align*}
where $(a)$ follows from our assumption $a \geq b$. Combining the two inequalities, we then have 
\begin{equation*}
\psi_T(u) \leq a+c \leq \psi_T(v),
\end{equation*}
which is one direction of the implication.

\begin{figure}
\begin{center}
\includegraphics[scale = 0.5]{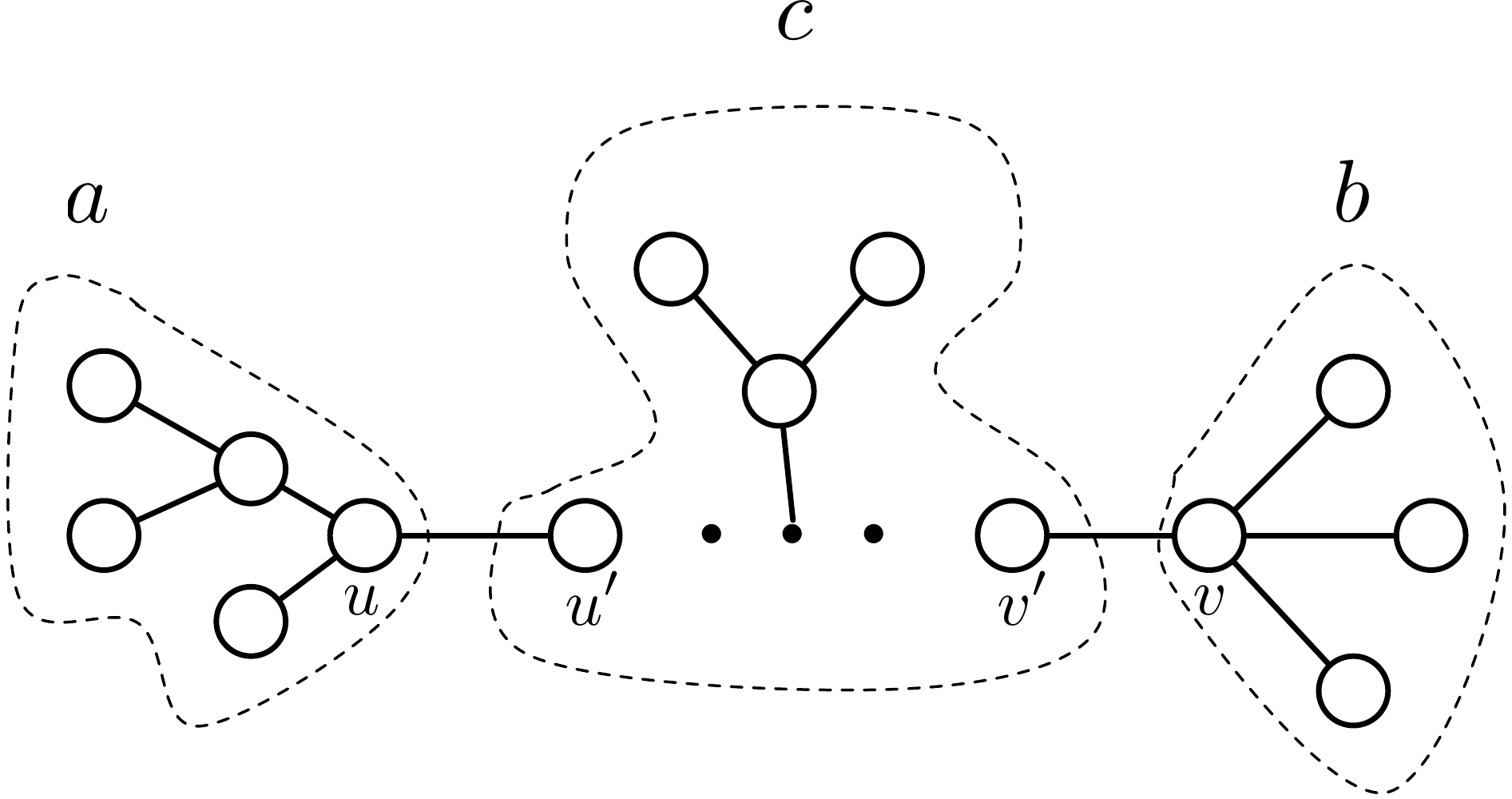}
\end{center}
\caption{Subtrees with sizes $a$, $b$, and $c$ from Lemma \ref{lemma: jogloh2.1}.}
\label{FigABC}
\end{figure}

If instead $a < b$, the same steps establish the string of inequalities
\begin{equation*}
\psi_T(v) \le \max(b-1, \; a+c) < b+c \le \psi_T(v),
\end{equation*}
providing the other direction of the implication.
\end{proof}


\section{Supporting proofs for Theorem~\ref{thm: adam}}\label{appendix: adam}

In this Appendix, we provide proofs of the lemmas used to derive Theorem~\ref{thm: adam}.

\subsection{Proof of Lemma~\ref{LemBhalu}}\label{appendix: LemBhalu}
\label{AppLemBhalu}

First note that we clearly have $\psi(v_1) \leq \max\left(|T^n_{1,2}|, |T^n_{2,2}|\right)$ and $n = |T^n_{1,2}| + |T^n_{2,2}|$. Thus,
\begin{align}
\label{EqnMalt}
\mathbb P(\psi(v_1) \geq (1-\delta)n) &\leq \mathbb P\left(\frac{\max\left(|T^n_{1,2}|, |T^n_{2,2}|\right)}{|T^n_{1,2}| + |T^n_{2,2}|} \geq (1-\delta)\right) \notag \\
&\leq P\left(\frac{|T^n_{1,2}|}{|T^n_{1,2}| + |T^n_{2,2}|} \geq (1-\delta)\right) \notag \\
& \qquad + P\left(\frac{|T^n_{2,2}|}{|T^n_{1,2}| + |T^n_{2,2}|} \geq (1-\delta)\right).
\end{align}

Consider the continuous time versions of the growing tree processes, and let $\theta$ be the Malthusian parameter of the point process associated with $T^t_{1,2}$. Then
\begin{equation*}
\frac{|T^t_{1,2}|}{|T^t_{1,2}|+|T^t_{2,2}|} = \frac{e^{-\theta t}|T^t_{1,2}|}{e^{-\theta t}|T^t_{1,2}|+e^{-\theta t}|T^t_{2,2}|} \stackrel{a.s.} \longrightarrow \frac{W_1}{W_1+W_2} := W,
\end{equation*}
where by Lemma~\ref{lem: sing}, the random variable $W$ is absolutely or singular continuous and is supported on the entire interval $[0,1]$. In particular, we may choose $\delta_0' > 0$ such that $\mathbb P(W \geq 1-\delta_0) < \frac{\epsilon}{4}$. This implies that
$$\lim\sup_{t \to \infty} \mprob\left(\frac{|T^t_{1,2}|}{|T^t_{1,2}| + |T^t_{2,2}|} \geq (1-\delta_0')\right) < \frac{\epsilon}{4}.$$
Using a similar argument for the second term, we conclude that there exists a $\delta_0''$ such that 
$$\lim\sup_{t \to \infty} \mprob\left(\frac{|T^t_{2,2}|}{|T^t_{1,2}| + |T^t_{2,2}|} \geq (1-\delta_0'')\right) < \frac{\epsilon}{4}.$$
Taking $\delta_0 = \min(\delta_0', \delta_0'')$ and substituting back into inequality~\eqref{EqnMalt}, we obtain the desired bound.

\subsection{Proof of Lemma~\ref{LemHouston}}\label{appendix: LemHouston}
\label{AppLemHouston}
As noted in the proof sketch, for any $i > K$, we have
$$\psi(v_i) \geq \min_{1 \leq k \leq K} \sum_{j = 1, j \neq k}^K |T^n_{j,K}|.$$
Hence,
\begin{equation}
\label{EqnHari}
\mathbb P(\exists i > K : \psi(v_i) \leq (1-\delta_0)n) \leq \mathbb P\left(\exists 1 \leq k \leq K : \sum_{j=1, j\neq k}^K |T^n_{j,K}| \leq (1-\delta_0)n\right).
\end{equation}
We can break up the right-hand expression as follows: From Theorem 22 in Bhamidi~\cite{Bhamidi07}, the maximum degree of a sublinear preferential attachment model with attraction function \mbox{$f(i) = (i+1)^\alpha$} scales as $(\log n)^{\frac{1}{1-\alpha}}$. Concretely, there exists a constant $M$ such that 
\begin{equation*}
\lim\sup_{n \to \infty} \mathbb P \left( \text{Max-Deg} (T_n) > ( \log n)^{\frac{1}{1-\alpha}}M\right) < \frac{\epsilon}{4}.
\end{equation*}
Therefore, we may choose $N$ large enough such that 
\begin{equation}\label{eq: k and N}
\mathbb P \left( \text{Max-Deg} (T_n) > (\log n)^{\frac{1}{1-\alpha}}M\right) < \frac{\epsilon}{4}, \quad \text{for all } n \geq N.
\end{equation}
Note that $M$ depends only on $\epsilon$ and the distribution of the normalized maximum degree that exists in the limit of the the $\alpha$-sublinear attachment tree growth process. Thus, fixing $\epsilon$ fixes $M$, as well. Having chosen $M$, note that $N$ depends on how fast the normalized distribution of the maximum degree converges to the fixed distribution, and on $\epsilon$. Since the former is solely a property of the sublinear attachment process, we observe that $N$ also depends only on $\epsilon$.
We now pick a value $K > N$, and define the event
\begin{equation*}
E_{K} \defn \left\{\text{Max-Deg}(T_{K}) \leq (\log K)^{\frac{1}{1-\alpha}} M\right\}.
\end{equation*}
The right-hand side of inequality~\eqref{EqnHari} may be bounded by
\begin{align*}
& \mathbb P\left(\exists 1 \leq k \leq K : \sum_{j=1, j\neq k}^{K} |T^n_{j,K}| \leq (1-\delta_0)n   ~\Big|~   E_{K}\right) \nonumber \\
&\qquad + \mathbb P \left(E_{K} \right)\\
&\stackrel{(a)} \leq \mprob \left(\exists 1 \leq k \leq K : \sum_{j=1, j\neq k}^{K} |T^n_{j,K}| \leq (1-\delta_0)n   ~\Big|~   E_{K} \right)\nonumber \\
&\qquad + \frac{\epsilon}{4}\\
&\stackrel{(b)}\leq \sum_{k=1}^{K} \mathbb P \left(\sum_{j=1, j\neq k}^{K} |T^n_{j,K}| \leq (1-\delta_0)n   ~\Big|~   E_{K} \right) + \frac{\epsilon}{4}.
\end{align*}
Here, $(a)$ follows from equation~\eqref{eq: k and N} and the choice of $K > N$. Step $(b)$ is a simple application of the union bound.
Now fix $k=1$, and consider the probability
\begin{align*}
\mprob \left(\sum_{j=2}^{K} |T^n_{j,K}| \leq (1-\delta_0)n ~\Big|~ E_{K} \right)
%
%
\stackrel{(a)} = \mathbb P \left(\sum_{j=2}^{K} |T^n_{j,K}| \leq \left(\frac{1-\delta_0}{\delta_0} \right)|T^n_{1,K}| ~\Big|~  E_{K} \right), \label{eq: T1K}
%
\end{align*}
where step $(a)$ follows since $\sum_{j=1}^K |T^n_{j, K}|$ is simply the total number of vertices, which is $n$.
Since the degree of $v_1$ is at most $(\log K)^{\frac{1}{1-\alpha}} M$ conditioned on $E_K$, we may bound the above probability via stochastic domination, as follows: At time $n = K$, replace $v_1$ by $\lceil (\log K)^{\frac{1}{1-\alpha}} M\rceil$ isolated vertices, and replace $v_j$ by a single isolated vertex, for each $2 \leq j \leq K$. The crucial step is to observe that by Lemma \ref{lemma: separation}, this replacement expedites the growth of $|T^t_{1,K}|$ and retards the growth of $\sum_{j=2}^K |T^t_{j,K}|$. Applying Lemma~\ref{lemma: poling} to the i.i.d.\ limit random variables $W_i$ and $\tilde W_i$ corresponding to the renormalized populations of the continuous time CMJ processes, we then have
\begin{align*}
&\lim\sup_{t \to \infty} \mathbb P \left(e^{-\theta t}\sum_{j=2}^K |T^t_{j,K}| \leq \left(\frac{1-\delta_0}{\delta_0} \right) e^{-\theta t} |T^t_{1,K}| ~\Big|~ E_K \right)\\
& \qquad \qquad \qquad \qquad \qquad \qquad \leq \mathbb P \left( \sum_{i=1}^{K-1} W_i \leq \left(\frac{1-\delta_0}{\delta_0}\right) \sum_{i=1}^{\lceil(\log K)^{\frac{1}{1-\alpha}} M \rceil} \tilde W_i\right)\\
& \qquad \qquad \qquad \qquad \qquad \qquad \leq \mathbb P \left( \sum_{i=1}^{K-1} W_i \leq   \tilde U \right),
%
\end{align*}
where  $\tilde U$ is the random variable $\left(\frac{1-\delta_0}{\delta_0}\right)\sum_{i=1}^{\lceil(\log K)^{\frac{1}{1-\alpha}} M\rceil} \tilde W_i$. In anticipation of using Lemma \ref{lemma: poling}, we bound $\mathbb E[\tilde U^2]$ as follows:
\begin{align*}
\mathbb E \left[\tilde U^2\right] &= \left(\frac{1-\delta_0}{\delta_0}\right)^2\mathbb E \left[\left(\sum_{i=1}^{\lceil(\log K)^{\frac{1}{1-\alpha}} M\rceil} \tilde W_i\right)^2\right]\\
&\stackrel{(a)}\leq  \left(\frac{1-\delta_0}{\delta_0}\right)^2\lceil(\log K)^{\frac{1}{1-\alpha}} M\rceil^2 \mathbb E \tilde W_1^2,
\end{align*}
where step $(a)$ is true because for all $1 \leq i, j \leq \lceil(\log K)^{\frac{1}{1-\alpha}} M\rceil$, we have $$\mathbb E[\tilde W_i \tilde W_j] = \mathbb E[\tilde W_i]^2 \leq \mathbb E[\tilde W_i^2].$$
Now we apply Lemma \ref{lemma: poling} to conclude that
\begin{align*}
P \left( \sum_{i=1}^{K-1} W_i \leq  \tilde U \right) &\leq  \frac{C\times \left(\frac{1-\delta_0}{\delta_0}\right)^2\lceil(\log K)^{\frac{1}{1-\alpha}} M\rceil^2 \mathbb E \tilde W_1^2}{(K-1)^2}\\
&= \frac{C_1(\log K)^{\frac{2}{1-\alpha}}}{(K-1)^2}, 
\end{align*} 
where the constant $C_1 = \left(\frac{1-\delta_0}{\delta_0}\right)^2 \times C \times M^2 \times \mathbb E \tilde W_1^2$ depends only on $\epsilon$, since by Lemma~\ref{lemma: poling}, the constant $C$ depends only on the distribution of $\tilde W_i$, which in turn depends only on the sublinear preferential attachment growth process and is therefore fixed. Arguing similarly, $\mathbb E \tilde W_1^2$ is again a fixed constant. Also, as noted earlier, $\delta_0$ and $M$ depend only on $\epsilon$. 

Since such an inequality holds for all values $1 \leq k \leq K$, substituting back into inequality~\eqref{EqnHari} and applying a union bound yields
\begin{equation*}
\lim\sup_{n\to \infty} \mathbb P\left(\exists i > K : \psi(v_i) \leq (1-\delta_0)n\right) \leq K\frac{C_1(\log K)^{\frac{2}{1-\alpha}}}{(K-1)^2} + \frac{\epsilon}{4}.
\end{equation*}
We now choose $K>N$ sufficiently large so that $\frac{C_1K(\log K)^{\frac{2}{1-\alpha}}}{(K-1)^2} < \frac{\epsilon}{4}$, establishing the desired inequality.

\section{Additional technical lemmas}
\label{appendix: wall}

In this Appendix, we state and prove a useful Hoeffding bound for sums of independent, nonnegative random variables.

\begin{lemma*} \label{lemma: poling}
Let $X_1, X_2, \dots, X_n$ be i.i.d.\ random variables distributed according to $Z$, such that $Z\geq 0$ almost surely and $\mathbb E [Z^2] < \infty$. Let $Y$ be a random variable independent of $X_i$'s satisfying $Y>0$ almost surely and $\mathbb E[Y^2] < \infty$. Then
\begin{equation*}
\mathbb P \left(\sum_{i=1}^n X_i \leq Y\right) \leq \frac{C \, \E[Y^2]}{n^2},
\end{equation*}
for some constant $C$ depending only on the distribution of $Z$.
\end{lemma*}

\begin{proof}
Define $W_i = \min\{X_i, M\}$, where the constant $M$ is chosen such that $\mathbb E [W_i] \geq \frac{\mathbb E [X_i]}{2}$. Since $W_i \leq X_i$, we have
\begin{align}
\label{EqnNM}
& \mathbb P\left( \frac{1}{n} \sum_{i=1}^n X_i - \mathbb E [X_i] \leq t\right) \leq \mathbb P\left( \frac{1}{n} \sum_{i=1}^n W_i - \mathbb E [X_i] \leq t\right) \notag \\
& \qquad \qquad \qquad \leq \mathbb P\left( \frac{1}{n} \sum_{i=1}^n W_i -  2\mathbb E [W_i] \leq t\right) \notag \\
& \qquad \qquad \qquad = \mathbb P\left( \frac{1}{n} \sum_{i=1}^n W_i -  \mathbb E [W_i] \leq t + \mathbb E [W_i]\right) \notag \\
& \qquad \qquad \qquad \leq C_1\exp(-nt^2C_2),
\end{align}
for suitable constants $C_1$ and $C_2$, where the last inequality follows from Hoeffding's inequality. Let
\begin{equation*}
E_1 := \left\{\frac{1}{n} \sum_{i=1}^n X_i - \mathbb E [X_i] \leq -\frac{\mathbb E [X_i]}{2}\right\}.
\end{equation*}
Then
\begin{align*}
\mathbb P\left( \sum_{i=1}^n X_i \leq Y\right) \leq \mathbb P(E_1) + \mathbb P\left(Y \geq \frac{n \mathbb E [X_i]}{2}\right).
\end{align*}
Note that by Markov's inequality, we have the bound
\begin{equation*}
\mathbb P\left(Y \geq \frac{n\mathbb E [X_i]}{2}\right) = \mathbb P\left(Y^2 \ge \frac{n^2 \E[X_i]^2}{4}\right) \leq \frac{4 \mathbb E [Y^2]}{n^2 \mathbb E [X_i]^2} = \frac{C_3 \, \E[Y^2]}{n^2},
\end{equation*}
for a suitable constant $C_3$. Since $\mathbb P(E_1)$ decays exponentially in $n$ by inequality~\eqref{EqnNM}, we may find another constant $C_4$ such that
\begin{align*}
\mathbb P\left(\sum_{i=1}^n X_i \le Y\right) \leq \frac{C_4 \, \E[Y^2]}{n^2},
\end{align*}
as claimed.
\end{proof}

\begin{lemma*} \label{lemma: separation}
Consider a CMJ process initiated from the following state: The root node gives birth according to a shifted point process $\xi_d$ such that
$$\mathbb P\left(\xi_d(t+dt) - \xi_d(t) = 1 \mid \xi_d(t)=i\right) = f(i+d)dt + o(dt),$$
where $f$ is the attraction function for an $\alpha$-sublinear preferential attachment process. Apart from the root node, every other node gives birth according the point process $\xi$ driven by the function $f(i) = (i+1)^\alpha$. Let the population of this CMJ process be denoted by $H(t)$. Let $X_i(t)$ for $1 \leq i \leq d+1$ be i.i.d.\ CMJ processes initiated from a single point. We claim that $\sum_{i=1}^{d+1} X_i(t)$ stochastically dominates $H(t)$.
\end{lemma*}

\begin{proof}
Consider the root node $v$ and is CMJ process $H(t)$. We compare its growth with the sum of $d+1$ i.i.d.\ CMJ processes starting from the isolated vertices $\{u_1, \dots, u_{d+1}\}$. Let $C_v(t)$ denote the number of children of $v$ at time $t$, and let $C_i(t)$ denote the number of children of $u_i$ at time $t$. Let $C_u(t) = \sum_{i=1}^{d+1} C_i(t)$. Note that $C_v(t)$ is simply a Markov process, given by
\begin{equation*}
\mathbb P\left(C_v(t+dt) - C_v(t) \mid C_v(t) = k\right) = (d+k+1)^\alpha \; dt +o(dt).
\end{equation*}
Unlike $C_v$, the process $C_u$ is not Markov. However, for any $(r_1, \dots, r_{d+1})$ such that $\sum_{i=1}^{d+1} r_i = k$, we may write
\begin{equation*}
\mathbb P\Bigg(C_u(t+dt) - C_u(t) = 1 \; \bigg | \; C_i(t) = r_i, \text{ for } 1 \leq i \leq d+1\Bigg) = \sum_{i=1}^{d+1} (r_i+1)^\alpha.
\end{equation*}
Since $\alpha < 1$, we see that no matter what the $r_i$'s are, we must have
\begin{equation*}
(d+k+1)^\alpha \leq \sum_{i=1}^{d+1} (r_i+1)^\alpha.
\end{equation*}
Thus, the process $C_2(t)$ stochastically dominates $C_1(t)$. Since the children in each process behave identically; i.e., they reproduce according to $\xi$, and so do their descendants, we can couple the processes in a straightforward way to conclude that the sum of $d+1$ independent CMJ processes stochastically dominates $H(t)$.
\end{proof}

\end{appendix}

\end{document}